\documentclass[12pt,a4paper]{article}
\usepackage[T1]{fontenc}
\usepackage[utf8]{inputenc}
\usepackage{lmodern}
\usepackage[body={16.6true cm, 26.4true cm}]{geometry}
\usepackage{amsmath,amssymb,amsthm}
\usepackage{microtype}

\usepackage{newtxtext,newtxmath}% Text and mathematics fonts.
\usepackage{indentfirst}
\usepackage{xcolor}
\usepackage[colorlinks=true,linkcolor=red,citecolor=blue]{hyperref}
\usepackage{mathtools}\mathtoolsset{showonlyrefs}

\numberwithin{equation}{section}
\newtheorem{theorem}{Theorem}[section]
\newtheorem{proposition}[theorem]{Proposition}
\newtheorem{lemma}[theorem]{Lemma}
\newtheorem{corollary}{Corollary}
\theoremstyle{remark}
\newtheorem{remark}[theorem]{Remark}
\newtheorem{example}[theorem]{Example}
\newcommand{\R}{\mathbb R}
\newcommand{\Z}{\mathbb Z}
\newcommand{\Om}{\Omega}
\newcommand{\cl}[1]{\overline{#1}}
\newcommand{\ind}{\mathbf 1}
\newcommand{\dd}{\,\mathrm d}
\newcommand{\ip}[2]{\langle #1,#2\rangle}
\DeclareMathOperator{\dist}{dist}
\DeclareMathOperator{\diam}{diam}
\DeclareMathOperator{\Lip}{Lip}

\DeclareMathOperator{\diver}{div}
\allowdisplaybreaks[2]

\title{Near-boundary rigidity and strict maximum modulus for Neumann eigenfunctions on convex planar domains\footnote{\footnotesize The work is supported by the National Natural Science Foundation of China (No.12001276, No.12326303, No.12431018).}}
\hypersetup{pdftitle={Near-Boundary Rigidity and Strict Maximum Modulus for Neumann Eigenfunctions on Convex Planar Domains}}

\author{Haiyun Deng$^{1}$, Xuyong Jiang$^{2}$, Xiaoping Yang$^{3}$\footnote{\footnotesize E-mail: hydeng@nau.edu.cn(H. Deng), jiangxy@cczu.edu.cn(X. Jiang), xpyang@nju.edu.cn(X. Yang)} \\[12pt]
	\emph {\scriptsize $^{1}$Department of Applied Mathematics, Nanjing Audit University, Nanjing, 211815, China;}\\
	\emph {\scriptsize $^{2}$Department of Mathematics, Changzhou University, Changzhou, 213164, China;}\\
	\emph {\scriptsize $^{3}$Department of Mathematics, Nanjing University, Nanjing, 210093, China} }
\date{}

\begin{document}
	\maketitle
	\begin{abstract}
		We develop a near-boundary rigidity theorem for arbitrary nonzero Neumann eigenfunctions on bounded convex planar domains. Precisely, if an eigenfunction $u$ with positive eigenvalue $\mu$ has an interior maximum-modulus point $p$ satisfying
		$\sqrt\mu\,\dist(p,\partial\Om)\le\pi/\sqrt3$, then $\Om$ must be a rectangle and $u$ must be a cosine of one coordinate, and  $\mu\ge4\mu_2(\Om)$, with the factor $4$ being sharp.
		This shows that a scale-invariant distance condition at just one maximum-modulus point determines the geometry of the entire domain
		and the explicit form of the eigenfunction.  We further prove that the maximum modulus of
		every second Neumann eigenfunction on a bounded convex planar domain
		is attained only on the boundary, without assuming boundary
		smoothness or eigenvalue simplicity. Moreover, we also prove strict boundary maximum modulus for higher Neumann eigenfunctions under suitable conditions, including all third Neumann eigenfunctions on arbitrary triangles.
	\end{abstract}
	
	
	
	\section{Introduction}\label{sec:introduction}

	Let $\Om\subset\R^2$ be a bounded convex domain. We consider real-valued
	eigenfunctions of the Neumann problem
	\begin{equation}\label{eq:neumann}
		\begin{cases}
			-\Delta u=\mu u & \text{in }\Om,\\
			\partial_\nu u=0 & \text{on }\partial\Om,
		\end{cases}
	\end{equation}
	where $\nu$ is the outward unit normal, defined almost everywhere on
	$\partial\Om$. Since the boundary need not be smooth, the equation is
	understood in the weak sense: $u\in H^1(\Om)$ satisfies
	\begin{equation}\label{eq:weak}
		\int_\Om\nabla u\cdot\nabla\varphi\dd x
		=\mu\int_\Om u\varphi\dd x,
		\qquad \varphi\in H^1(\Om).
	\end{equation}
	The eigenvalues, counted with multiplicity, are denoted by
	\begin{equation}\label{eq:eigenvalue-order}
		0=\mu_1(\Om)<\mu_2(\Om)\le\mu_3(\Om)\le\cdots.
	\end{equation}
	We write $E_\mu(\Om)$ for the eigenspace corresponding to $\mu$, and
	$E_2(\Om)=E_{\mu_2(\Om)}(\Om)$.

	Rauch's hot spots conjecture \cite{Rauch} asks whether a second Neumann
	eigenfunction attains its maximum and minimum on the boundary. Its strict
	form asserts that neither extremal value is attained in the interior.
	The geometry of the domain is essential: counterexamples are known on
	multiply connected planar domains \cite{BurdzyWerner} and on convex
	domains in sufficiently high dimensions \cite{DeDiosPont}.

	{For planar domains, the hot spots conjecture has been proved for several
		classes, including convex domains with two axes of symmetry \cite{JerisonNadirashvili},
		lip domains \cite{AtarBurdzy}, nearly circular convex domains \cite{Miyamoto},
		and triangles \cite{JudgeMondal,JudgeMondalErratum}.
		Further results concern critical points and monotonicity on triangles
		\cite{ChenGuiYao}, and symmetry, multiplicity, and critical points on
		isosceles trapezoids and kites \cite{DengEtAlKites}.
		Methods used in these works include probabilistic coupling \cite{BanuelosBurdzy},
		domain deformation, and nodal set analysis.}
	Recently, OpenAI \cite{SmoothHotSpots2026} stated a theorem excluding
	interior critical points of every second Neumann eigenfunction
	on smooth bounded simply connected planar domains, without a simplicity
	assumption. Taking that result as an input,
	Burdzy \cite{BurdzyBoundary2026} derived the non-strict hot-spots property
	on bounded simply connected planar domains with compact Neumann resolvent
	and simple second eigenvalue, without assuming boundary smoothness.
     Atar and Burdzy \cite{AtarBurdzy2} construct a bounded
starlike Lipschitz domain in $\mathbb{R}^{3}$ with simple first positive Neumann eigenvalue for which the corresponding eigenfunction has its unique global maximum at an interior point.
    For quantitative comparisons between
	interior and boundary extrema, see \cite{DeDiosPontHsuTaylor,DengJiangYang2026,MarianoPanzoWang,SteinerbergerConstant}
	and the references therein.
    The arguments in the present paper do not rely on the methods of \cite{SmoothHotSpots2026} and \cite{BurdzyBoundary2026}. Moreover, our approach yields results beyond the scope of those works, including a near-boundary rigidity theorem for arbitrary positive Neumann eigenvalues and strict boundary maximum-modulus criteria for higher Neumann eigenfunctions.

	In this paper we first establish a rigidity theorem for arbitrary positive Neumann eigenvalues on bounded convex planar domains. The existence
	of just one interior maximum-modulus point suitably close to the boundary is enough to determine the geometry of the entire domain and the explicit form of the eigenfunction.

	\begin{theorem}\label{thm:rigidity}
		Let $u$ be a nonzero Neumann eigenfunction of a bounded convex domain
		$\Om\subset\R^2$, with eigenvalue $\mu>0$, and set
		$M=\max_{\cl\Om}|u|$. Suppose that $p\in\Om$ satisfies $|u(p)|=M$ and
		\begin{equation}\label{eq:threshold}
			\sqrt\mu\,\dist(p,\partial\Om)\le\frac\pi{\sqrt3}.
		\end{equation}
		Then $\Om$ is a rectangle and, for a unit vector $\tau$ parallel to one
		of its sides,
		\begin{equation}\label{eq:rigid-function}
			u(x)=\pm M\cos\bigl(\sqrt\mu\,\ip{x-p}{\tau}\bigr).
		\end{equation}
		The sides parallel to $\tau$ have length $m\pi/\sqrt\mu$ for some integer
		$m\ge2$. In particular,
		\begin{equation}\label{eq:spectral-factor}
			\mu\ge4\mu_2(\Om).
		\end{equation}
	\end{theorem}
	
	The rectangular examples in Remark {\ref{rem:sharp-factor}} show
	that the factor $4$ in {\eqref{eq:spectral-factor}} is sharp.
	As a consequence of Theorem {\ref{thm:rigidity}}, we obtain a strict boundary inequality for the maximum modulus of second Neumann eigenfunctions on arbitrary bounded convex planar domains, without assuming boundary smoothness or simplicity of $\mu_2(\Om)$.
	
	\begin{theorem}\label{thm:main}
		Let $u$ be a second Neumann eigenfunction of a bounded convex
		domain $\Om\subset\R^2$. Then
		\begin{equation}\label{eq:main}
			|u(x)|<\max_{\partial\Om}|u|
			\qquad\text{for every }x\in\Om.
		\end{equation}
		In particular, $\max_{\cl\Om}|u|=\max_{\partial\Om}|u|$.
	\end{theorem}
	For centrally symmetric planar domains, Kennedy \cite[Theorem 1.3]{Kennedy} proved that every second Neumann eigenfunction
	is odd under central inversion, assuming boundedness, simple
	connectedness, and a Lipschitz boundary. Combining this symmetry with
	Theorem \ref{thm:main} gives the following conclusion on centrally symmetric
	convex domains.

	\begin{corollary}\label{cor:central}
		Let $u$ be a second Neumann eigenfunction of a bounded convex
		domain $\Om\subset\R^2$ centrally symmetric about $c\in\R^2$.
		Then
		\begin{equation}\label{eq:central-oddness}
			u(2c-x)=-u(x)\qquad(x\in\cl\Om).
		\end{equation}
		In particular, its maximum and minimum are attained only on $\partial\Om$.
	\end{corollary}  

    The strict boundary maximum modulus property also extends to higher Neumann eigenfunctions under a suitable condition,
showing that it is not restricted to second Neumann eigenfunctions.
Writing $r_\Om = \max_{x\in\Om}\dist(x,\partial\Om)$ for the inradius of $\Om$, we obtain the following theorem.

	\begin{theorem}\label{thm:higher-maximum-modulus}
		Let $\Om\subset\R^2$ be a bounded convex domain that is not a rectangle.
		If $\mu>0$ is a Neumann eigenvalue satisfying
		\begin{equation}\label{eq:inradius-criterion}
			\mu r_\Om^2\le\frac{\pi^2}{3},
		\end{equation}
		then every $0\ne u\in E_\mu(\Om)$ satisfies
		\begin{equation}\label{eq:higher-maximum-modulus}
			|u(x)|<\max_{\partial\Om}|u|
			\qquad\text{for every }x\in\Om.
		\end{equation}
	\end{theorem}

	Theorems \ref{thm:main} and \ref{thm:higher-maximum-modulus} control the larger of $\max_{\cl\Om}u$ and $-\min_{\cl\Om}u$, but the extremum with smaller absolute value may still be attained in the interior. In Example \ref{ex:sector-smaller-extremum}, a third Neumann eigenfunction satisfies \eqref{eq:inradius-criterion} and has
	$-\min_{\cl\Om}u<\max_{\cl\Om}u$, yet its minimum is attained
	in the interior.

	Combining Theorem \ref{thm:higher-maximum-modulus} with Kröger’s diameter bound \cite{Kroger},
	we obtain a geometric condition ensuring strict boundary maximum modulus
	for every eigenfunction corresponding to the first \(N\) positive Neumann eigenvalues.
	Writing $D_\Om=\diam\Om$, we obtain the following corollary. Here $j_{0,1}$ denotes the first positive zero of the Bessel function $J_0$.
	
	\begin{corollary}\label{cor:finite-higher-modes}
		Let $\Om\subset\R^2$ be a bounded convex domain that is not a rectangle,
		and let $N\ge1$ be an integer. If
		\begin{equation}\label{eq:diameter-inradius-criterion}
			\frac{D_\Om}{r_\Om}
			\ge\sqrt3\left(N-1+\frac{2j_{0,1}}{\pi}\right),
		\end{equation}
		then every eigenfunction associated with any of
		$\mu_2(\Om),\ldots,\mu_{N+1}(\Om)$ satisfies the strict
		maximum-modulus inequality \eqref{eq:higher-maximum-modulus}.
	\end{corollary}

	For triangles, Theorem \ref{thm:higher-maximum-modulus} applies to every third Neumann eigenfunction. To the best of our knowledge,
	the strict boundary maximum modulus property for third Neumann eigenfunctions
	on arbitrary triangles has not previously been established.
	
	\begin{theorem}\label{thm:triangle-third-mode}
		Let $u$ be a third Neumann eigenfunction of a triangle $T\subset\R^2$. Then
		\begin{equation}\label{eq:triangle-third-maximum-modulus}
			|u(x)|<\max_{\partial T}|u|
			\qquad\text{for every }x\in T.
		\end{equation}
		Moreover, let $S=|T|$, $P$ the perimeter of $T$, and $D=\diam T$
		the length of its longest side. Set
		$\rho=2S/(P-D)$. If $T$ is acute and
		\begin{equation}\label{eq:triangle-vertex-criterion}
			\mu_3(T)\rho^2\le\frac{\pi^2}{3},
		\end{equation}
		then $|u|$ attains its maximum only at the vertices of $T$.
	\end{theorem}
	
	Judge and Mondal \cite{JudgeMondalCriticalSets} studied critical sets of Laplace eigenfunctions.
	They proved that a second Neumann
	eigenfunction on a bounded simply connected polygon has only finitely many
	critical points in its closure unless the polygon is a rectangle and the
	eigenfunction is a cosine of one coordinate. The following theorem gives a
	related classification for the interior global maximum-modulus set, for arbitrary
	positive Neumann eigenvalues on bounded convex planar domains.
	
	\begin{corollary}
		\label{thm:maximum-modulus-set}
		Let $u$ be a nonzero Neumann eigenfunction of a bounded
		convex domain $\Om\subset\R^2$, with eigenvalue $\mu>0$. Set
		\begin{equation}\label{eq:interior-maximum-modulus-set}
			\mathcal M_{\mathrm{int}}(u)=\{x\in\Om:|u(x)|= \max_{\cl\Om}|u| \}.
		\end{equation}
		{The set $\mathcal M_{\mathrm{int}}(u)$ is infinite if and only if,}
		after a rigid motion, there exist $W,L>0$, $C\in\R\setminus\{0\}$,
		and an integer $m\ge2$ such that
		\begin{equation}\label{eq:maximum-set-rectangular-mode}
			\Om=(0,W)\times(0,L),\qquad
			u(x_1,x_2)=C\cos\frac{m\pi x_2}{L},\qquad
			\mu=\frac{m^2\pi^2}{L^2}.
		\end{equation}
		In every other case, $\mathcal M_{\mathrm{int}}(u)$ is finite, possibly
		empty, and
		\begin{equation}\label{eq:maximum-set-distance-gap}
			\mathcal M_{\mathrm{int}}(u)
			\subset
			\left\{x\in\Om:\dist(x,\partial\Om)
			\ge\frac{\pi}{\sqrt{3\mu}}\right\}.
		\end{equation}
	\end{corollary}

		The main ingredient of Theorem \ref{thm:rigidity}  is a rigidity mechanism that turns an interior
		maximum-modulus point sufficiently close to the boundary into a
		one-dimensional eigenfunction and forces the domain to be a rectangle.
		The proof of Theorem \ref{thm:rigidity} consists of three main steps.
		
		First, combining a perturbed Li--Yau type maximum-principle argument
		with smooth convex approximation, we establish the sharp gradient estimate
		\[
		|\nabla u|^2+\mu u^2\le \mu M^2,
		\qquad M=\max_{\overline{\Omega}}|u|,
		\]
		on arbitrary bounded convex planar domains. This estimate yields a
		cosine lower bound centered at any interior maximum point.
		
		Second, we develop a comparison argument that propagates an interior
		maximum to a line segment. After normalization, let $u(p)=M=1$,
		$\mu=1$, and $d=\operatorname{dist}(p,\partial\Omega)$.
		Using convexity and a nearest boundary point, we construct a
		comparison region $D$ and compare $u$ with the model solution
		$\cos x_2$. The essential difficulty is that the comparison
		function does not satisfy the Neumann condition on the boundary of $D$, so a direct maximum-principle argument
		is unavailable. We overcome this by introducing the positive part
		$z=(\cos x_2-u)^+$ and the auxiliary function
		\[
		h(x_1,x_2)=
		\cos\left(\frac{\pi x_1}{2d}\right)
		\cos\left(\frac{x_2}{2}\right).
		\]
		Combining a comparison identity with a quantitative inequality
		for supporting normals, we control the boundary contributions and
		obtain
		\[
		\left(\frac{\pi^2}{4d^2}-\frac34\right)
		\int_D zh\,dx\le0.
		\]
		This forces $z=0$ when $d<\pi/\sqrt3$.
		At the critical threshold $d=\pi/\sqrt3$, where the coefficient
		vanishes, we use the equality cases of the sharp one-dimensional
		Poincar\'e inequalities to obtain the same conclusion.
		Consequently, $u$ attains its maximum along the entire segment
		joining $p$ to a nearest boundary point.
		
		Third, we convert this maximum segment into global rigidity.
		Analytic uniqueness for the Helmholtz equation implies that
		$u=\cos x_2$ throughout the normalized domain.
		To recover the geometry without assuming boundary smoothness,
		we first test the weak Neumann equation with functions $q(x_2)$
		to show that the horizontal cross-sectional width is constant.
		Convexity then makes the lateral sides straight, and testing with
		$x_1q(x_2)$ shows that these sides are parallel to the $x_2$-axis.
		Thus the domain is a rectangle, and its side in the oscillation
		direction contains at least two half-periods.
		This yields the spectral separation
		\[
		\mu\ge4\mu_2(\Omega).
		\]

	The rest of this paper is organized as follows. Section \ref{sec:estimates} contains the approximation argument and
	the gradient estimate. Section \ref{sec:comparison}  shows how
	the maximum propagates from a point to a segment. Section \ref{sec:rigidity} derives rectangular rigidity
	and completes the proofs of the main results.

	\section{Domain perturbation and gradient estimates}
	\label{sec:estimates}
	
	Since a bounded convex domain has Lipschitz boundary, 
	$H^1(\Om)\hookrightarrow L^2(\Om)$ is compact. The Neumann Laplacian therefore has discrete spectrum and an $L^2(\Om)$-orthonormal basis of eigenfunctions, see \cite{Brezis,Grisvard}. The min--max principle gives
	
	\begin{equation}\label{eq:rayleigh}
		\mu_2(\Om)=
		\min_{\substack{0\ne v\in H^1(\Om)\\ \int_\Om v=0}}
		\frac{\int_\Om|\nabla v|^2\dd x}{\int_\Om v^2\dd x}.
	\end{equation}
	More generally,
	\begin{equation}\label{eq:minmax-all}
		\mu_j(\Om)=\min_{\substack{V\subset H^1(\Om)\\ \dim V=j}}
		\max_{0\ne v\in V}
		\frac{\int_\Om|\nabla v|^2\dd x}{\int_\Om v^2\dd x}.
	\end{equation}
	Interior elliptic regularity implies that a weak eigenfunction is smooth
	and real analytic in $\Om$.

	The outward unit normal is denoted by $\nu$, wherever it exists. If $\Om$ is a bounded convex domain,
	$\varphi\in\Lip(\cl \Om)$, and $\psi\in C^2(\R^2)$, then
	\begin{equation}\label{eq:green}
		\int_\Om(\nabla\varphi\cdot\nabla\psi+\varphi\Delta\psi)\dd x
		=\int_{\partial \Om}\varphi\,\partial_\nu\psi\dd s.
	\end{equation}

	\subsection{The gradient estimate on smooth convex domains}
	We use a Li–Yau type maximum-principle argument to obtain a pointwise gradient estimate
	for Neumann eigenfunctions on smooth convex planar domains, see also \cite{LiYau,XuXue} for eigenfunction gradient estimates.

	\begin{lemma}\label{lem:perturbed-gradient}
		Let $K\subset\R^2$ be a bounded smooth convex domain and let
		$w\in C^\infty(\cl K)$ satisfy $\partial_\nu w=0$.
		Fix $a>\mu>0$ and put
		\begin{equation}\label{eq:P-definitions}
			f=\Delta w+\mu w,\qquad P=|\nabla w|^2+aw^2,
			\qquad M_w=\max_{\cl K}|w|,\quad G_w=\max_{\cl K}|\nabla w|.
		\end{equation}
		Then
		\begin{equation}\label{eq:perturbed}
			|\nabla w|^2+aw^2
			\le aM_w^2+
			\frac{(6a-4\mu)M_w\|f\|_\infty+
				2G_w\|\nabla f\|_\infty}{2(a-\mu)}.
		\end{equation}
		All norms in this formula are over $\cl K$.
	\end{lemma}
	
	\begin{proof}
		Bochner's identity in the plane, together with $\Delta w=f-\mu w$, gives
		\begin{equation}\label{eq:P-general}
			\Delta P=2|D^2w|^2+2(a-\mu)|\nabla w|^2
			+2\nabla w\cdot\nabla f+2awf-2a\mu w^2.
		\end{equation}
		At a point where $\nabla P=0$ and $\nabla w\ne0$, the identity
		$\nabla P=2D^2w\nabla w+2aw\nabla w$ gives
		\begin{equation}\label{eq:hessian-direction}
			D^2w\,\nabla w=-aw\nabla w.
		\end{equation}
		Hence $-aw$ is an eigenvalue of $D^2w$, and the other eigenvalue is
		$\Delta w+aw=(a-\mu)w+f$. Therefore
		\begin{equation}\label{eq:hessian-square}
			|D^2w|^2=a^2w^2+\bigl((a-\mu)w+f\bigr)^2.
		\end{equation}
		Substitution into \eqref{eq:P-general} yields
		\begin{equation}\label{eq:P-identity}
			\begin{split}
				\Delta P={}&2(a-\mu)\bigl((2a-\mu)w^2+|\nabla w|^2\bigr)\\
				&+(6a-4\mu)wf+2f^2+2\nabla w\cdot\nabla f.
			\end{split}
		\end{equation}
		
		Let $\tau$ be a unit tangent on $\partial K$, and use the curvature
		convention $D_\tau\nu=\kappa\tau$, so $\kappa\ge0$ for a convex boundary.
		Tangential differentiation of $\nabla w\cdot\nu=0$ gives
		\begin{equation}\label{eq:boundary-hessian}
			D^2w(\tau,\nu)=-\kappa\partial_\tau w.
		\end{equation}
		Since $\nabla w=(\partial_\tau w)\tau$ on $\partial K$, we obtain
		\begin{equation}\label{eq:boundary-P}
			\partial_\nu P=-2\kappa(\partial_\tau w)^2\le0.
		\end{equation}
		Let $q\in\cl K$ be a maximum point of $P$. Suppose first that
		$q\in\partial K$. For sufficiently small $s>0$, the point $q-s\nu(q)$
		lies in $K$, and $P(q-s\nu(q))\le P(q)$. Its one-sided derivative at
		$s=0$ is therefore nonpositive, so $\partial_\nu P(q)\ge0$.
		Together with \eqref{eq:boundary-P}, this gives $\partial_\nu P(q)=0$.
		The tangential derivative also vanishes, hence $\nabla P(q)=0$.
		Along the same inward normal ray the first derivative is now zero,
		so the one-sided maximum gives
		$D^2P(q)(-\nu,-\nu)=D^2P(q)(\nu,\nu)\le0$.
		If $\beta$ is a unit-speed boundary curve with $\beta(0)=q$ and
		$\beta'(0)=\tau$, the chain rule gives
		$(P\circ\beta)''(0)=D^2P(q)(\tau,\tau)+\nabla P(q)\cdot\beta''(0)$.
		The last term vanishes, and the boundary maximum makes this second
		derivative nonpositive. Adding the normal and tangential inequalities
		gives $\Delta P(q)\le0$. The usual interior maximum argument yields
		the same properties when $q\in K$. Thus in both cases
		\begin{equation}\label{eq:max-P-properties}
			\nabla P(q)=0,\qquad \Delta P(q)\le0.
		\end{equation}
		If $\nabla w(q)=0$, then $P(q)=aw(q)^2\le aM_w^2$.
		Otherwise apply \eqref{eq:P-identity} at $q$ and discard the nonnegative
		terms involving $w(q)^2$ and $f(q)^2$. This gives
		\begin{equation}\label{eq:maximum-gradient-bound}
			2(a-\mu)|\nabla w(q)|^2
			\le (6a-4\mu)M_w\|f\|_\infty+2G_w\|\nabla f\|_\infty.
		\end{equation}
		Adding $aw(q)^2\le aM_w^2$ gives the bound for $P(q)$. Since $q$ is
		a maximum point, the same bound holds throughout $\cl K$, which is \eqref{eq:perturbed}.
	\end{proof}
	
	\begin{corollary}\label{cor:smooth-gradient}
		If $w$ is a Neumann eigenfunction on a bounded smooth convex planar
		domain $K$ with eigenvalue $\mu>0$, then
		\begin{equation}\label{eq:smooth-gradient}
			|\nabla w|^2+\mu w^2\le \mu \max_{\cl K}|w|^2.
		\end{equation}
	\end{corollary}
	
	\begin{proof}
		For each $a>\mu$, Lemma \ref{lem:perturbed-gradient} with $f=0$
		gives $|\nabla w|^2+aw^2\le aM_w^2$. Letting $a\to\mu$
		proves the assertion.
	\end{proof}
	
	\subsection{Smooth approximation and spectral convergence}
	
	We approximate the domain from outside and establish convergence of
	each eigenspace.
	
	\begin{lemma}\label{lem:outer-approximation}
		By translating the bounded convex domain $\Omega$, we may assume $0\in\Omega$.
		Then we can choose smooth bounded convex domains $\Omega_k$ such that
		\begin{equation}\label{eq:approximation}
			\cl\Om\subset\Om_k\subset(1+\varepsilon_k)\Om,
			\qquad \varepsilon_k \to 0.
		\end{equation}
		Moreover, the following statements hold.
		\begin{enumerate}
			\item For every fixed $j$, $\mu_j(\Om_k)\to\mu_j(\Om)$.
			\item One can choose $L^2(\Om_k)$-orthonormal eigenbases
			$\{\phi_j^k\}_{j\ge1}$ and a subsequence such that,
			for every fixed $j$, $\phi_j^k|_{\cl\Om}$ converges uniformly
			to a Lipschitz eigenfunction $\psi_j$ on $\cl\Om$ with
			eigenvalue $\mu_j(\Om)$.
			\item If $\mu$ has multiplicity $m$ and
			$\mu_J(\Om)=\cdots=\mu_{J+m-1}(\Om)=\mu$,
			then the corresponding limits $\psi_J,\ldots,\psi_{J+m-1}$
			form an orthonormal basis of $E_\mu(\Om)$.
		\end{enumerate}
	\end{lemma}
	
	\begin{proof}
		Choose $r_0>0$ with $B(0,r_0)\subset\Om$, and put $C=\cl\Om$.
		The distance function $d_C(x)=\dist(x,C)$ is convex and $1$-Lipschitz.
		Let $\eta_\delta$ be a nonnegative smooth mollifier of integral $1$,
		supported in $B(0,\delta)$, and set $F_\delta=d_C*\eta_\delta$.
		Then $F_\delta$ is smooth and convex, and
		\begin{equation}\label{eq:mollified-distance}
			|F_\delta-d_C|\le\delta,\qquad
			F_\delta\le\delta\text{ on }C.
		\end{equation}
		Define $K_\delta=\{F_\delta<2\delta\}$. Since $F_\delta\le\delta$ on $C$,
		we have $C\subset K_\delta$. The convexity of $F_\delta$ implies that
		$K_\delta$ is convex. For $x\in K_\delta$,
		$d_C(x)\le F_\delta(x)+\delta<3\delta$, so
		$K_\delta\subset C+3\delta B(0,1)$ and $K_\delta$ is bounded.
		Moreover, we claim that $\partial K_\delta$ is smooth.
		In fact,
		suppose that $F_\delta(x)=2\delta$ and $\nabla F_\delta(x)=0$. 
		Since a critical point of a differentiable convex function
		is a global minimum,
		$F_\delta$ would have global minimum value $2\delta$, contradicting $F_\delta\le\delta$ on $C$.
		Hence $\nabla F_\delta\ne 0$ on $\{ F_\delta = 2\delta\}$,
		and the implicit function theorem shows that $\partial K_\delta$ is smooth.
		
		For a unit vector $n$, define the support function
		$h_C(n)=\max_{y\in C}y\cdot n$. Since $B(0,r_0)\subset\Om$,
		we have $h_C(n)\ge r_0$. Fix $\varepsilon>0$ and choose $\delta>0$
		with $3\delta<\varepsilon r_0$. Every $x\in C+3\delta B(0,1)$
		satisfies $x\cdot n\le h_C(n)+3\delta$ for every unit $n$, and hence
		\[
		(1+\varepsilon)h_C(n)-x\cdot n
		\ge\varepsilon r_0-3\delta=:a>0.
		\]
		For $y\in B(x,a)$, the inequality
		$y\cdot n\le x\cdot n+|y-x|<(1+\varepsilon)h_C(n)$
		holds for every unit $n$. A closed convex set is the intersection of
		its supporting closed half-planes, so $B(x,a)\subset(1+\varepsilon)C$.
		Thus $x$ is an interior point of $(1+\varepsilon)C$.
		Since $\operatorname{int}C=\Om$, we obtain
		$C+3\delta B(0,1)\subset(1+\varepsilon)\Om$, and consequently
		$K_\delta\subset(1+\varepsilon)\Om$.
		Choose $\varepsilon_k=1/k$, $\delta_k=r_0/(6k)$, and
		$\Om_k=K_{\delta_k}$. Then $3\delta_k<\varepsilon_k r_0$, which
		proves \eqref{eq:approximation}. All these domains contain $B(0,r_0)$
		and lie in $2\Om$, so their diameters are bounded by a common constant
		$D_0>2\diam\Om$.

		Let $V_j$ be the span of the first $j$ eigenfunctions on $\Om$.
		For $v\in V_j$ set
		\begin{equation}\label{eq:scaled-test}
			T_kv(x)=v\bigl(x/(1+\varepsilon_k)\bigr),\qquad x\in\Om_k,
			\qquad A_k=(1+\varepsilon_k)^{-1}\Om_k.
		\end{equation}
		The inclusions give
		$(1+\varepsilon_k)^{-1}\Om\subset A_k\subset\Om$.
		In dimension two, a change of variables gives
		\begin{equation}\label{eq:scaled-rayleigh}
			\frac{\int_{\Om_k}|\nabla T_kv|^2}{\int_{\Om_k}|T_kv|^2}
			=\frac{1}{(1+\varepsilon_k)^2}
			\frac{\int_{A_k}|\nabla v|^2}{\int_{A_k}v^2}.
		\end{equation}
		Choose an $L^2(\Om)$-orthonormal basis $v_1,\ldots,v_j$ of $V_j$.
		Since $|\Om\setminus A_k|\to0$, absolute continuity of the integral gives
		$\eta_k:=\sum_{a=1}^j\int_{\Om\setminus A_k}v_a^2\dd x\to0$.
		For $v=\sum_{a=1}^jc_av_a$, the Cauchy--Schwarz inequality gives
		$v^2\le(\sum_{a=1}^jc_a^2)\sum_{a=1}^jv_a^2$ pointwise.
		Integrating over $\Om\setminus A_k$ and using orthonormality, we obtain
		\[
		\int_{A_k}v^2\dd x\ge(1-\eta_k)\int_\Om v^2\dd x
		\qquad(v\in V_j).
		\]
		For all sufficiently large $k$, we have $\eta_k<1$, so a change of
		variables shows that $T_kv\ne0$ whenever $0\ne v\in V_j$.
		Thus $T_k$ is injective on $V_j$, and $T_kV_j$ is a $j$-dimensional
		subspace of $H^1(\Om_k)$.
		We may therefore apply the min--max characterization
		\eqref{eq:minmax-all} to the $j$-dimensional space $T_kV_j$
		to obtain an upper bound for $\mu_j(\Om_k)$.
		Since $A_k\subset\Om$ and $V_j$ is spanned by the first $j$
		eigenfunctions on $\Om$, we have
		\[
		\int_{A_k}|\nabla v|^2\dd x
		\le\int_\Om|\nabla v|^2\dd x
		\le\mu_j(\Om)\int_\Om v^2\dd x
		\qquad(v\in V_j).
		\]
		Combining these estimates with \eqref{eq:scaled-rayleigh},
		and applying \eqref{eq:minmax-all} with $T_kV_j$ gives
		\[
		\mu_j(\Om_k)
		\le\sup_{0\ne v\in V_j}
		\frac{\int_{\Om_k}|\nabla T_kv|^2\dd x}
		{\int_{\Om_k}|T_kv|^2\dd x}
		\le\frac{\mu_j(\Om)}{(1+\varepsilon_k)^2(1-\eta_k)}.
		\]
		Letting $k\to\infty$, we conclude that

		\begin{equation}\label{eq:limsup-spectrum}
			\limsup_{k\to\infty}\mu_j(\Om_k)\le\mu_j(\Om).
		\end{equation}

		Fix $\Lambda>0$. We next obtain uniform bounds for normalized
		eigenfunctions with eigenvalue at most $\Lambda$.
		For an $L^2(\Omega_k)$-normalized eigenfunction $\phi$ with eigenvalue at most $\Lambda$,
		Corollary \ref{cor:smooth-gradient} gives
		$|\nabla\phi|\le\sqrt{\Lambda} A$, where $A=\max_{\cl{\Omega_k}}|\phi|$.
		In particular, convexity also gives $\Lip(\phi)\le \sqrt{\Lambda} A$.
		Both bounds also hold for the constant eigenfunction with eigenvalue zero.
		Choose $q$ where $|\phi(q)|=A$ and put
		\begin{equation}\label{eq:contracted-ball}
			\theta=\min\left\{\frac12,\frac1{2\sqrt\Lambda D_0}\right\},
			\qquad B_q=(1-\theta)q+\theta B(0,r_0).
		\end{equation}
		This open ball lies in $\Om_k$ even if $q\in\partial\Om_k$, and every
		$x\in B_q$ satisfies $|x-q|\le\theta D_0$.
		Hence the Lipschitz bound gives $|\phi(x)|\ge A/2$ and
		\begin{equation}\label{eq:sup-bound-normalized}
			1=\int_{\Om_k}\phi^2\dd x
			\ge\frac{A^2}{4}\pi\theta^2r_0^2.
		\end{equation}
		Thus both $\|\phi\|_{L^\infty}$ and $\Lip(\phi)$ are bounded
		uniformly in $k$ for all normalized eigenfunctions with eigenvalue at most $\Lambda$.
		
		Write $\mu_j^k=\mu_j(\Om_k)$ and choose an orthonormal eigenbasis
		$\{\phi_j^k\}_{j\ge1}$ on each domain.
		For each fixed $j$, \eqref{eq:limsup-spectrum} bounds $\mu_j^k$, and the preceding estimates give uniform $L^\infty$ and Lipschitz bounds for $\phi_j^k$. Applying the Arzel\`a--Ascoli theorem and weak compactness
		in $H^1(\Om)$, and taking a diagonal subsequence over $j$, we obtain a subsequence such that, for every fixed $j$, we have
		\begin{equation}\label{eq:basis-limits}
			\mu_j^k\to\lambda_j\le\mu_j(\Om),\qquad
			\phi_j^k|_{\cl\Om}\to\psi_j\quad\text{uniformly},
			\qquad\nabla\phi_j^k\rightharpoonup\nabla\psi_j\text{ in }L^2(\Om),
		\end{equation}
		and each $\psi_j$ is Lipschitz. For fixed $i,j$, the uniform supremum
		bounds and $|\Om_k\setminus\Om|\to0$ imply
		$\int_{\Om_k\setminus\Om}\phi_i^k\phi_j^k\dd x\to0$.
		Combining this with uniform convergence on $\cl\Om$ gives
		\begin{equation}\label{eq:limit-orthonormal}
			\int_\Om\psi_i\psi_j\dd x
			=\lim_{k\to\infty}\int_{\Om_k}\phi_i^k\phi_j^k\dd x=\delta_{ij}.
		\end{equation}
		To pass to the weak equation, first take an ambient test function
		$\zeta\in C_c^\infty(\R^2)$. The terms over $\Om_k\setminus\Om$ vanish
		because $\phi_j^k$ and $\nabla\phi_j^k$ are uniformly bounded there.
		Weak convergence of the gradients and uniform convergence of the functions
		on $\Om$ therefore give
		\begin{equation}\label{eq:limit-weak}
			\int_\Om\nabla\psi_j\cdot\nabla\zeta\dd x
			=\lambda_j\int_\Om\psi_j\zeta\dd x.
		\end{equation}
		Restrictions of functions in $C_c^\infty(\R^2)$ are dense in
		$H^1(\Om)$ because $\Om$ is a bounded Lipschitz domain
		\cite{Brezis,Grisvard}. Both sides of \eqref{eq:limit-weak} are
		continuous in the $H^1$ norm, so the equality holds for every
		$\zeta\in H^1(\Om)$.
		Thus $\psi_j$ is a weak Neumann eigenfunction with eigenvalue $\lambda_j$.

		For $i\le j$, the inequalities $\mu_i^k\le\mu_j^k$ hold for
		every $k$. Passing to the limit gives $\lambda_i\le\lambda_j$.
		On the span of
		$\psi_1,\ldots,\psi_j$, orthonormality and \eqref{eq:limit-weak} give
		\begin{equation}\label{eq:limit-rayleigh}
			\frac{\int_\Om|\nabla\sum_{i=1}^jc_i\psi_i|^2}
			{\int_\Om|\sum_{i=1}^jc_i\psi_i|^2}
			=\frac{\sum_{i=1}^j\lambda_i c_i^2}{\sum_{i=1}^jc_i^2}
			\le\lambda_j.
		\end{equation}
		The min--max principle now implies $\mu_j(\Om)\le\lambda_j$.
		Together with \eqref{eq:basis-limits}, this gives
		$\lambda_j=\mu_j(\Om)$. Starting from any subsequence of the original
		domains, the same argument produces a further subsequence with this
		eigenvalue limit. Hence $\mu_j(\Om_k)\to\mu_j(\Om)$ along the full sequence.
		
		Let $\mu$ be an eigenvalue of multiplicity $m$, and let $J$ be its first
		index in the ordered sequence of eigenvalues. Thus
		\begin{equation}\label{eq:multiple-eigenvalue}
			\mu_J(\Om)=\cdots=\mu_{J+m-1}(\Om)=\mu.
		\end{equation}
		The corresponding limits are $m$ orthonormal elements of the
		$m$-dimensional eigenspace $E_\mu(\Om)$ and hence form a basis.
	\end{proof}
	
	\subsection{The gradient estimate on a general convex domain}
	
	\begin{proposition}\label{prop:gradient}
		Let $u$ be a nonzero weak Neumann eigenfunction on a bounded convex planar
		domain $\Om$, with eigenvalue $\mu>0$. Then $u$ has a Lipschitz representative
		on $\cl\Om$.
		Set $M=\max_{\cl\Om}|u|$, we have
		\begin{equation}\label{eq:gradient}
			|\nabla u|^2+\mu u^2\le\mu M^2\qquad\text{in }\Om,
		\end{equation}
		and
		\begin{equation}\label{eq:phase}
			\left|\arccos\frac{u(x)}M-\arccos\frac{u(y)}M\right|
			\le\sqrt\mu\,|x-y|\qquad(x,y\in\cl\Om).
		\end{equation}
		Consequently, if $u(p)=M$, then
		\begin{equation}\label{eq:barrier}
			u(x)\ge M\cos\bigl(\min\{\sqrt\mu\,|x-p|,\pi\}\bigr)
			\qquad(x\in\cl\Om).
		\end{equation}
	\end{proposition}
	
	\begin{proof}
		Let $\mu$ have multiplicity $m$, and let $J$ be its first index, as in
		\eqref{eq:multiple-eigenvalue}. By Lemma \ref{lem:outer-approximation},
		$\psi_J,\ldots,\psi_{J+m-1}$ form an orthonormal basis of $E_\mu(\Om)$.
		Write $u=\sum_{i=J}^{J+m-1}c_i\psi_i$ and set
		\begin{equation}\label{eq:eigenfunction-approximation}
			w_k=\sum_{i=J}^{J+m-1}c_i\phi_i^k,\qquad
			f_k=\Delta w_k+\mu w_k
			=\sum_{i=J}^{J+m-1}c_i(\mu-\mu_i^k)\phi_i^k.
		\end{equation}
		The finite sum defining $u$ provides its Lipschitz representative.
		The function $w_k$ satisfies the Neumann condition on $\partial\Om_k$,
		although it need not be an exact eigenfunction. The uniform bounds in the proof of
		Lemma \ref{lem:outer-approximation} give
		\begin{equation}\label{eq:approximation-convergence}
			\begin{gathered}
				w_k\to u\quad\text{uniformly on }\cl\Om,\qquad
				\sup_k\|\nabla w_k\|_{L^\infty(\Om_k)}<\infty,\\
				\|f_k\|_{L^\infty(\Om_k)}+
				\|\nabla f_k\|_{L^\infty(\Om_k)}\longrightarrow0.
			\end{gathered}
		\end{equation}

		Let $M_k=\max_{\cl{\Om_k}}|w_k|$. The lower limit is at least $M$, since
		$\Om\subset\Om_k$ and convergence is uniform on $\cl\Om$.
		Conversely, if $x\in\cl{\Om_k}$, then
		$y=x/(1+\varepsilon_k)\in\cl\Om$. If
		$R_0=\max_{\cl\Om}|y|$, then $|x-y|\le\varepsilon_k R_0$.
		The common Lipschitz bound for $w_k$ gives
		\begin{equation}\label{eq:maxima-convergence}
			M_k\le\max_{\cl\Om}|w_k|+C\varepsilon_k R_0,
			\qquad\text{and therefore }M_k\to M.
		\end{equation}
		
		Fix $a>\mu$ and $\delta>0$. By Lemma \ref{lem:perturbed-gradient}, the
		second term in \eqref{eq:perturbed} tends to zero. For all sufficiently
		large $k$ it is at most $a\delta$, so
		\begin{equation}\label{eq:regularized-gradient}
			|\nabla w_k|^2\le a(M_k^2+\delta-w_k^2).
		\end{equation}
		Since $M_k^2+\delta-w_k^2\ge\delta$, the function
		\begin{equation}\label{eq:regularized-phase}
			\Phi_k=\arccos\frac{w_k}{\sqrt{M_k^2+\delta}}
		\end{equation}
		is smooth and satisfies $|\nabla\Phi_k|\le\sqrt a$. Integrating on segments in the convex
		set $\Om_k$ gives
		\begin{equation}\label{eq:regularized-phase-distance}
			\left|\arccos\frac{w_k(x)}{\sqrt{M_k^2+\delta}}
			-\arccos\frac{w_k(y)}{\sqrt{M_k^2+\delta}}\right|
			\le\sqrt a\,|x-y| \qquad(x,y\in\cl\Om).
		\end{equation}
		First let $k\to\infty$, using
		\eqref{eq:approximation-convergence} and \eqref{eq:maxima-convergence}.
		Next let $\delta\to0$, and finally let $a\to\mu$.
		Continuity of $\arccos$ on $[-1,1]$ justifies the second limit,
		even where $u=\pm M$. This proves \eqref{eq:phase}.

		At an interior point with $|u|<M$, 
		\begin{equation}\label{eq:phase-differentiation}
			\left|\nabla\arccos\frac uM\right|^2
			=\frac{|\nabla u|^2}{M^2-u^2}\le\mu.
		\end{equation}
		This proves \eqref{eq:gradient} there. At an interior point with $|u|=M$,
		the smooth function $u$ is at a local extremum, so $\nabla u=0$ and the
		same inequality holds. Finally, if $u(p)=M$, \eqref{eq:phase} gives
		$\arccos(u(x)/M)\le\sqrt\mu|x-p|$. For distances at most $\pi/\sqrt\mu$,
		the monotonicity of cosine on $[0,\pi]$ gives the desired
		lower bound. For larger distances, use $u(x)\ge-M$.
		This proves \eqref{eq:barrier}.
	\end{proof}

	\section{Propagation of the maximum from a point to a segment}
	\label{sec:comparison}
	
	\subsection{Supporting geometry and notation}
	
	Throughout this section, $u$ is a weak Neumann eigenfunction on the bounded convex planar domain $\Om$ normalized by
	\begin{equation}\label{eq:normalized-problem}
		\mu=1,\qquad u(p)=\max_{\cl\Om}|u|=1,\qquad p\in\Om.
	\end{equation}
	Put $d=\dist(p,\partial\Om)>0$ and choose
	a nearest boundary point $x_0$.
	
	For $x=(x_1,x_2)\in\partial\Om$, an outward supporting unit normal at $x$
	is a vector $n=(n_1,n_2)\in\R^2$ with $|n|=1$ such that
	\[
	(y-x)\cdot n\le0\qquad\text{for every }y\in\cl\Om.
	\]
	Equivalently, the line $\{y:(y-x)\cdot n=0\}$ supports $\Om$ at $x$,
	and $n$ points into the half-plane opposite to $\Om$.
	Convexity ensures the existence of such a normal at every boundary point.
	At a smooth boundary point the supporting normal is unique and
	agrees with the outward unit normal $\nu$. At a nonsmooth point it need
	not be unique. The outward unit normal $\nu$ exists almost everywhere
	on $\partial\Om$. We write $\nu=(\nu_1,\nu_2)$,
	where $\nu_i$ denotes the component of $\nu$ in the $x_i$-direction.
	\begin{lemma}\label{lem:supporting-geometry}
		After a rigid motion, we have
		\begin{equation}\label{eq:configuration}
			x_0=0,\qquad p=(d,0),\qquad
			B(p,d)\subset\Om\subset\{x_1>0\}.
		\end{equation}
		For any $x\in\partial\Om$ and any outward supporting unit normal $n$ at $x$, we have
		\begin{equation}\label{eq:normal}
			x_2n_2\ge d+(d-x_1)n_1.
		\end{equation}
		In particular, if $0\le x_1\le d$, then $x_2n_2\ge x_1$.
	\end{lemma}
	
	\begin{proof}
		The open disk $B(p,d)$ lies in $\Om$ by the definition of distance to the
		boundary, and its closure lies in $\cl\Om$. A supporting line to $\Om$ at
		$x_0$ also supports this closed disk at $x_0$. The disk has a unique tangent
		line there, perpendicular to $p-x_0$. 
		Translating $x_0$ to the origin and taking $p-x_0$ as the positive $x_1$-direction gives \eqref{eq:configuration}.

		For a supporting normal $n$ at any boundary point $x$, every
		$y\in\cl\Om$ satisfies $(y-x)\cdot n\le0$. Apply this to
		$y=p+dn\in\cl{B(p,d)}$. Since $|n|=1$,
		\begin{equation}\label{eq:normal-calculation}
			0\ge(p+dn-x)\cdot n
			=(d-x_1)n_1-x_2n_2+d.
		\end{equation}
		This proves \eqref{eq:normal}. If $0\le x_1\le d$, then
		$d-x_1\ge0$ and $n_1\ge-1$, so
		$d+(d-x_1)n_1\ge d-(d-x_1)=x_1$. Thus $x_2n_2\ge x_1$.
	\end{proof}
	
	Proposition \ref{prop:gradient} gives
	\begin{equation}\label{eq:normalized-barrier}
		-1\le u\le1,\qquad
		u(x)\ge\cos\bigl(\min\{|x-p|,\pi\}\bigr)
		\qquad(x\in\cl\Om).
	\end{equation}
	Set
	\begin{equation}\label{eq:comparison-sets}
		S=(0,d)\times(-\pi,\pi),\qquad D=\Om\cap S,
		\qquad\Lambda=\partial\Om\cap S.
	\end{equation}
	Define the remaining boundary pieces by
	\begin{equation}\label{eq:boundary-partition}
		\begin{split}
			\Lambda_0&=\cl\Om\cap(\{0\}\times(-\pi,\pi)),\\
			\Gamma&=\cl\Om\cap
			\bigl((\{d\}\times[-\pi,\pi])\cup([0,d]\times\{-\pi,\pi\})\bigr).
		\end{split}
	\end{equation}
	The point $(d/2,0)$ belongs to $B(p,d)\cap S$, so $D$ is
	nonempty. It is convex as the intersection of two convex sets.
	For any $x\in\cl\Om\cap\cl S$, the segment from $x$ to $(d/2,0)$,
	apart from $x$ itself, lies in $D$. Letting points on this segment
	approach $x$ gives $\cl D=\cl\Om\cap\cl S$.
	The normal on $\Lambda$ is the normal of $\Om$, and on $\Lambda_0$ it is
	$(-1,0)$ almost everywhere.

	\subsection{The comparison argument}
	Set $\gamma=\pi/(2d)$. We establish the following boundary inequality needed in the proof of Proposition \ref{prop:comparison}.
	
	\begin{lemma}\label{lem:boundary}
		For $x\in\Lambda$ and any outward supporting unit normal $n$ at $x$,
		put $z(x)=(\cos x_2-u(x))^+$. Then
		\begin{equation}\label{eq:boundary}
			z(x)\left(\gamma\tan(\gamma x_1)n_1
			+\frac12\tan\frac{x_2}{2} n_2\right)
			<\sin(x_2)n_2.
		\end{equation}
	\end{lemma}
	
	\begin{proof}
		Lemma \ref{lem:supporting-geometry} gives $x_2n_2\ge x_1>0$ on $\Lambda$.
		Thus $x_2$ and $n_2$ have the same nonzero sign. Write
		\begin{equation}\label{eq:boundary-variables}
			t=|x_2|\in(0,\pi),\qquad m=|n_2|>0.
		\end{equation}
		Then $\sin(x_2)n_2=\sin t\,m>0$ and
		$\tan(x_2/2)n_2=\tan(t/2)m$. Since $u\ge-1$, we have
		$0\le z\le1+\cos t$. Consequently,
		\begin{equation}\label{eq:normal-t}
			z\,\frac12\tan\frac t2\,m
			\le(1+\cos t)\frac12\tan\frac t2\,m
			=\frac12\sin t\,m.
		\end{equation}
		If $n_1\le0$, the first term in \eqref{eq:boundary} is nonpositive.
		By \eqref{eq:normal-t}, the left-hand side of
		\eqref{eq:boundary} is at most $\tfrac12\sin t\,m<\sin t\,m$,
		proving \eqref{eq:boundary} in this case.
		
		We now assume $n_1>0$, set $s=d-x_1\in(0,d)$ and
		$\rho=\min\{\sqrt{t^2+s^2},\pi\}$. The barrier
		\eqref{eq:normalized-barrier} gives $u(x)\ge\cos\rho$. Since $t\le\rho\le\pi$, we have
		\begin{equation}\label{eq:first-deficit}
			0\le z\le\cos t-\cos\rho.
		\end{equation}
		For $F(\sigma)=-\cos\sqrt\sigma$ and $0<\sigma\le\pi^2$, differentiation gives
		\begin{equation}\label{eq:concavity-calculation}
			F'(\sigma)=\frac{\sin\sqrt\sigma}{2\sqrt\sigma},\qquad
			F''(\sigma)=
			\frac{\sqrt\sigma\cos\sqrt\sigma-\sin\sqrt\sigma}{4\sigma^{3/2}} \le 0.
		\end{equation}
		Here the last inequality follows from the elementary inequality $t\cos t\le \sin t$ for $0\le t\le\pi$.
		By continuity at $0$, $F$ is concave on $[0,\pi^2]$.
		The tangent line inequality for concave functions at $t^2$ yields
		\begin{equation}\label{eq:deficit}
			\begin{split}
				0\le z
				&\le F(\rho^2)-F(t^2)
				\le F'(t^2)(\rho^2-t^2)\\
				&\le\frac{s^2\sin t}{2t}.
			\end{split}
		\end{equation}
		The last inequality also covers the truncated case $t^2+s^2>\pi^2$, since
		then $\rho^2-t^2=\pi^2-t^2\le s^2$.
		
		The identity $\gamma d=\pi/2$ gives
		\begin{equation}\label{eq:cotangent-bound}
			\gamma\tan(\gamma x_1)=\gamma\cot(\gamma s)\le\frac1s.
		\end{equation}
		Here the last inequality follows from the elementary inequality $y\cot y\le1$ for $0<y<\pi/2$.
		Moreover, \eqref{eq:normal} gives $tm\ge d+sn_1$, so in particular
		$sn_1<tm$. Therefore \eqref{eq:deficit} and \eqref{eq:cotangent-bound} give
		\begin{equation}\label{eq:normal-s}
			z\gamma\tan(\gamma x_1)n_1
			\le\frac{s n_1\sin t}{2t}
			<\frac12\sin t\,m.
		\end{equation}
		Adding \eqref{eq:normal-t} and \eqref{eq:normal-s} proves the lemma.
	\end{proof}
	
	The following proposition is a key step in the proof of rigidity. It shows that the distance condition forces $u$ to attain its maximum along the entire segment joining $p$ to a nearest boundary point.
	
	\begin{proposition}\label{prop:comparison}
		In the configuration \eqref{eq:configuration}, if $0<d\le\pi/\sqrt3$, then
		\begin{equation}\label{eq:comparison}
			u(x)\ge\cos x_2\qquad(x\in\cl D).
		\end{equation}
		In particular, $u(x_1,0)=1$ for $0\le x_1\le d$.
	\end{proposition}
	
	\begin{proof}
		Define
		\begin{equation}\label{eq:comparison-functions}
			\upsilon =\cos x_2-u,\qquad z=\upsilon^+,\qquad
			h=\cos(\gamma x_1)\cos\frac{x_2}{2}.
		\end{equation}
		The functions $\upsilon,z$ are Lipschitz on $\cl D$, and $h$ is smooth. A direct calculation gives
		\begin{equation}\label{eq:h}
			-\Delta h=\left(\gamma^2+\frac14\right)h,
			\quad h>0\text{ in }S,\quad h|_\Gamma=0,
			\quad\partial_1h|_{\{x_1=0\}}=0.
		\end{equation}
		On $x_1=d$, we have $|x-p|=|x_2|\le\pi$, so the barrier \eqref{eq:normalized-barrier} gives
		$u(d,x_2)\ge\cos x_2$. On $x_2=\pm\pi$, we have $u\ge-1=\cos x_2$.
		Consequently,
		\begin{equation}\label{eq:z-artificial}
			\upsilon \le0,\qquad z=0\qquad\text{on }\Gamma.
		\end{equation}
		
		Let $\varphi\in\Lip(\cl D)$ vanish on $\Gamma$, and extend it
		by zero from $D$ to $\Om$. To see that the extension is Lipschitz,
		take $x\in D$ and $y\in\Om\setminus D$. Convexity keeps the segment
		$[x,y]$ in $\Om$. Its first exit point $z_0$ from $S$ lies on
		$\Gamma$: the segment cannot meet $x_1=0$ because $\Om\subset\{x_1>0\}$.
		Since $\varphi(z_0)=0$,
		$|\varphi(x)|\le\Lip(\varphi)|x-z_0|\le\Lip(\varphi)|x-y|$.
		Pairs lying on the same side of $\partial D$ satisfy the Lipschitz
		bound directly. Thus the extension belongs to $H^1(\Om)$, with
		weak gradient $\ind_D\nabla\varphi$ almost everywhere, and is an
		admissible test function in \eqref{eq:weak}. Hence
		\begin{equation}\label{eq:u-local-test}
			\int_D(\nabla u\cdot\nabla\varphi-u\varphi)\dd x=0.
		\end{equation}
		Apply \eqref{eq:green} to the smooth function $\cos x_2$, which solves
		$(\Delta+1)\cos x_2=0$ and has normal derivative $-\sin(x_2)\nu_2$.
		The $\Gamma$ term vanishes because $\varphi=0$, and the $\Lambda_0$ term
		vanishes because $\nu_2=0$ almost everywhere there. Subtracting \eqref{eq:u-local-test} gives
		\begin{equation}\label{eq:v-weak}
			\int_D(\nabla \upsilon \cdot\nabla\varphi-\upsilon \varphi)\dd x
			=-\int_\Lambda\varphi\sin(x_2)\nu_2\dd s.
		\end{equation}

		For $\varepsilon>0$, let
		\begin{equation}\label{eq:positive-tests}
			\chi_\varepsilon=\min\{1,z/\varepsilon\},\qquad
			\varphi_\varepsilon=h\chi_\varepsilon.
		\end{equation}
		These functions are Lipschitz on $\cl D$ and vanish on $\Gamma$, so they
		are legitimate test functions. The Lipschitz chain rule yields
		\begin{equation}\label{eq:positive-test-gradient}
			\nabla \upsilon\cdot\nabla\varphi_\varepsilon
			=\chi_\varepsilon\nabla \upsilon\cdot\nabla h
			+\frac h\varepsilon\ind_{\{0<\upsilon<\varepsilon\}}|\nabla \upsilon|^2
			\ge\chi_\varepsilon\nabla \upsilon\cdot\nabla h
		\end{equation}
		almost everywhere. Thus \eqref{eq:v-weak} implies
		\begin{equation}\label{eq:positive-prelimit}
			\int_D\chi_\varepsilon(\nabla \upsilon\cdot\nabla h- \upsilon h)\dd x
			\le-\int_\Lambda h\chi_\varepsilon\sin(x_2)\nu_2\dd s.
		\end{equation}
		As $\varepsilon\to 0$, $\chi_\varepsilon\to\ind_{\{\upsilon>0\}}$.
		The volume integrands are dominated by
		$|\nabla \upsilon||\nabla h|+|\upsilon h|\in L^1(D)$.
		On $\Lambda$, the absolute value of the boundary integrand is
		at most $h$, which is bounded and integrable on the boundary of the
		bounded convex domain $D$. Dominated convergence therefore applies
		to both integrals. Using the chain rule
		$\nabla z=\ind_{\{\upsilon>0\}}\nabla\upsilon$ almost everywhere,
		we obtain
		\begin{equation}\label{eq:positive-weak}
			\int_D(\nabla z\cdot\nabla h-zh)\dd x
			\le-\int_{\Lambda\cap\{\upsilon>0\}}h\sin(x_2)\nu_2\dd s.
		\end{equation}

		Another application of \eqref{eq:green}, now to $z$ and the smooth function
		$h$, yields
		\begin{equation}\label{eq:green-h}
			\int_D(\nabla z\cdot\nabla h-zh)\dd x
			=\left(\gamma^2-\frac34\right)\int_Dzh\dd x
			+\int_\Lambda z\partial_\nu h\dd s.
		\end{equation}
		Here the $\Gamma$ term vanishes by \eqref{eq:z-artificial}, and the $\Lambda_0$
		term vanishes because $\partial_\nu h=0$ almost everywhere
		there. On $\Lambda$, the identity
		\begin{equation}\label{eq:normal-h}
			\partial_\nu h=-h\left(\gamma\tan(\gamma x_1)\nu_1
			+\frac12\tan\frac{x_2}{2}\,\nu_2\right),
		\end{equation}
		holds almost everywhere.
		The products $h\gamma\tan(\gamma x_1)$ and
		$\tfrac12h\tan(x_2/2)$ equal $-\partial_1h$ and $-\partial_2h$,
		respectively. These derivatives are bounded on $\cl S$, so the
		boundary integrals are finite even as $x_1\to d$ or $x_2\to\pm\pi$.
		
		Combining \eqref{eq:positive-weak}, \eqref{eq:green-h}, and
		\eqref{eq:normal-h}, using $z=0$ where $\upsilon\le0$, and applying
		Lemma \ref{lem:boundary}, we obtain
		\begin{equation}\label{eq:comparison-final}
			\begin{split}
				\left(\gamma^2-\frac34\right)\int_Dzh\dd x
				&\le\int_{\Lambda\cap\{\upsilon>0\}}h
				\left[z\left(\gamma\tan(\gamma x_1)\nu_1
				+\frac12\tan\frac{x_2}{2}\,\nu_2\right)
				-\sin(x_2)\nu_2\right]\dd s\\
				&\le0.
			\end{split}
		\end{equation}
		If $d<\pi/\sqrt3$, then
		$\gamma^2-3/4>0$. Since $z\ge0$ and $h>0$ in $D$, it follows that
		$\int_Dzh=0$. Continuity then forces $z=0$ on $\cl D$, which proves \eqref{eq:comparison}.
		
		We now consider the endpoint $d=\pi/\sqrt3$.
		Here $\gamma^2=3/4$, so Lemma
		\ref{lem:boundary} and
		\eqref{eq:comparison-final} imply that
		$\{\upsilon>0\}\cap\Lambda$ has boundary measure zero.
		Hence $z=0$ almost everywhere on $\Lambda$.
		Extend $z$ by zero from $D$ to $S$, keeping the same notation.
		Since $z$ vanishes almost everywhere on the internal interface $\partial\Om\cap S$, its zero extension belongs to $H^1(S)$.
		It also vanishes on $x_1=d$ and $x_2=\pm\pi$.
		Taking $\varphi=z$ in \eqref{eq:v-weak} gives
		\[
		\int_S|\nabla z|^2\dd x=\int_D|\nabla z|^2\dd x
		=\int_Dz^2\dd x=\int_Sz^2\dd x.
		\]
		The one-dimensional Neumann--Dirichlet and Dirichlet Poincar\'e
		inequalities give
		\[
		\begin{aligned}
			\int_0^d|\partial_1z(x_1,x_2)|^2\dd x_1
			&\ge\gamma^2\int_0^d z(x_1,x_2)^2\dd x_1,\\
			\int_{-\pi}^{\pi}|\partial_2z(x_1,x_2)|^2\dd x_2
			&\ge\frac14\int_{-\pi}^{\pi}z(x_1,x_2)^2\dd x_2,
		\end{aligned}
		\]
		for almost every $x_2$ and $x_1$, respectively.
		Therefore,
		\[
		\int_S|\nabla z|^2\dd x
		\ge\left(\gamma^2+\frac14\right)\int_Sz^2\dd x
		=\int_Sz^2\dd x.
		\]   
		Equality in both one-dimensional inequalities forces
		$z=c\cos(\gamma x_1)\cos(x_2/2)$ for some $c\ge0$.
		If $c>0$, then $z(x_1,0)=c\cos(\gamma x_1)$ and hence
		$
		u(x_1,0)=1-c\cos(\gamma x_1).
		$
		Thus
		$\partial_1u(p)=c\gamma\sin(\gamma d)=c\gamma>0$.
		This contradicts $\partial_1u(p)=0$, since $p=(d,0)$ is an interior
		maximum point of $u$. Thus $c=0$, and $z=0$ on $\cl D$
		in the endpoint case as well.
		
		For $0\le x_1\le d$, we have $(x_1,0)\in\cl D$, so $u(x_1,0)\ge\cos0=1$.
		The normalization gives the reverse inequality.
	\end{proof}
	
	\section{Rectangular rigidity and the main results}
	\label{sec:rigidity}
	
	\subsection{From a maximum segment to a rectangle}
	
	\begin{lemma}\label{lem:segment-analytic}
		Let $u$ solve $(\Delta+1)u=0$ in a connected planar domain containing
		$\Sigma=(0,d)\times\{0\}$. Suppose that $u\le1$ in the domain and
		$u(t,0)=1$ for $0<t<d$. Then $u(x)=\cos x_2$ throughout the domain.
	\end{lemma}
	
	\begin{proof}
		Every point of $\Sigma$ is an interior maximum, so $\nabla u=0$ on
		$\Sigma$. Therefore $w=u-\cos x_2$ satisfies
		\begin{equation}\label{eq:segment-cauchy}
			(\Delta+1)w=0,\qquad w(t,0)=0,
			\qquad\partial_2w(t,0)=0\quad(0<t<d).
		\end{equation}
		The functions $w(\cdot,0)$ and $\partial_2w(\cdot,0)$ vanish identically
		on an open interval, so all of their tangential derivatives vanish there.
		For each nonnegative integer $b$, differentiate $\partial_1^2 w+\partial_2^2 w+w=0$
		$b$ times with respect to $x_2$.
		Since $w$ is smooth, the partial derivatives commute, and rearranging then gives
		\begin{equation}\label{eq:derivative-recurrence}
			\partial_2^{b+2}w=-\partial_1^2\partial_2^bw-\partial_2^bw.
		\end{equation}
		The two initial functions $w(\cdot,0)$ and $\partial_2w(\cdot,0)$
		vanish. If $\partial_2^bw(\cdot,0)$ vanishes identically, then its
		second tangential derivative also vanishes, and the recurrence gives
		$\partial_2^{b+2}w(\cdot,0)=0$. Induction separately over the even
		and odd integers proves this for every nonnegative integer $b$.
		Tangential differentiation then gives
		\begin{equation}\label{eq:all-jets}
			\partial_1^a\partial_2^bw(t,0)=0
			\qquad(a,b\ge0,\ 0<t<d).
		\end{equation}
		At any fixed point of $\Sigma$, the convergent analytic Taylor series of
		$w$ is therefore zero. Hence $w$ vanishes on an open neighborhood of that
		point. Analytic continuation on the connected domain gives $w\equiv0$.
	\end{proof}

	\begin{lemma}\label{lem:rectangle}
		Let $u$ be a weak Neumann eigenfunction on a bounded convex planar
		domain $\Om$. Assume \eqref{eq:configuration}, $\mu=1$,
		$\max_{\cl\Om}|u|=1$, and
		$u(t,0)=1$ for $0<t<d$. Then
		\begin{equation}\label{eq:rectangle}
			u(x)=\cos x_2,\qquad
			\Om=(0,\omega_0)\times(-k\pi,l\pi),
		\end{equation}
		where $k,l\ge1$ are integers and $\omega_0\ge2d$.
		In particular, $\mu_2(\Om)\le1/4$.
	\end{lemma}
	
	\begin{proof}
		Lemma \ref{lem:segment-analytic} proves $u=\cos x_2$ in $\Om$.
		We now determine the domain using only the weak equation.
		The projection of $\Om$ onto the $x_2$-axis is an open bounded interval
		$(t_-,t_+)$ with $t_-<0<t_+$.
		For $t\in(t_-,t_+)$, convexity and openness imply
		\begin{equation}\label{eq:horizontal-slices}
			\{x_1:(x_1,t)\in\Om\}=(f(t),g(t))
			=(f(t),f(t)+\omega(t)).
		\end{equation}
		Fix $t_1,t_2\in(t_-,t_+)$ and $0<\theta<1$.
		Choose points in these two sections approaching their left endpoints.
		Convexity of $\Om$ keeps the segment joining the points inside $\Om$.
		Passing to the limit gives
		$f((1-\theta)t_1+\theta t_2)\le(1-\theta)f(t_1)+\theta f(t_2)$,
		so $f$ is convex. Applying the same argument to the right endpoints shows
		that $g$ is concave.
		Finite convex and concave functions are locally Lipschitz on open
		intervals. Thus $f$ and $g$ are locally Lipschitz, and in particular
		continuous on $(t_-,t_+)$. Both are bounded because $\Om$ is bounded.
		Thus the width $\omega=g-f$ is positive and continuous.

		Let $q$ be any Lipschitz function on $[t_-,t_+]$, and use
		$\varphi(x)=q(x_2)$ in \eqref{eq:weak}. Since
		$\nabla u=(0,-\sin x_2)$, Fubini's theorem gives
		\begin{equation}\label{eq:width}
			0=\int_{t_-}^{t_+}\bigl(q'(t)\sin t+q(t)\cos t\bigr)\omega(t)\dd t
			=\int_{t_-}^{t_+}(q\sin)'(t)\omega(t)\dd t.
		\end{equation}
		Let $I$ be a component of $(t_-,t_+)\setminus\pi\Z$.
		For any $F(t)\in C_c^\infty(I)$, choose $q(t)=F(t)/\sin(t)$ on $I$ and extend it by
		zero outside $I$. This is an admissible Lipschitz function because its
		support stays away from the zeros of $\sin(t)$. Equation \eqref{eq:width}
		then reads
		\begin{equation}\label{eq:width-distribution}
			\int_I F'(t)\omega(t)\dd t=0.
		\end{equation}
		Therefore $\omega'=0$ in distributions on $I$, so $\omega$ is constant
		there. The interval $(t_-,t_+)$ contains only finitely many multiples of
		$\pi$, and continuity of $\omega$ identifies the constants on adjacent
		components. Consequently, $\omega(t)\equiv\omega_0>0$ on the entire interval.
		
		With constant width, \eqref{eq:width} becomes
		\begin{equation}\label{eq:width-endpoints}
			0=\omega_0\bigl(q(t_+)\sin t_+-q(t_-)\sin t_-\bigr).
		\end{equation}
		Choose an affine $q$ with endpoint values $(q(t_-),q(t_+))=(0,1)$ and
		then one with endpoint values $(1,0)$. It follows separately that
		$\sin t_+=0$ and $\sin t_-=0$. Since $t_-<0<t_+$,
		\begin{equation}\label{eq:vertical-endpoints}
			t_-=-k\pi,\qquad t_+=l\pi,\qquad k,l\in\{1,2,\ldots\}.
		\end{equation}
		Both integers are positive because the maximum line $x_2=0$ crosses the
		interior.
		
		The function $f$ is convex, while $f+\omega_0$ is concave. Subtracting the
		constant shows that $f$ is also concave. Hence
		\begin{equation}\label{eq:affine-slice}
			f(t)=\alpha+\beta t.
		\end{equation}
		It remains to show that $\beta=0$.
		Take $\varphi(x)=x_1q(x_2)$ in \eqref{eq:weak}. This is admissible since
		$x_1$ is bounded on $\Om$ and $q$ is Lipschitz. After integrating first with respect to $x_1$ and using
		\begin{equation}\label{eq:first-slice-moment}
			\int_{\alpha+\beta t}^{\alpha+\beta t+\omega_0}x_1\dd x_1
			=\omega_0\left(\alpha+\beta t+\frac{\omega_0}{2}\right),
		\end{equation}
		we divide the identity by \(\omega_0\) and integrate by parts. The endpoint terms vanish because $\sin t_-=\sin t_+=0$, yielding
		\begin{equation}\label{eq:shear-test}
			0=\int_{t_-}^{t_+}(q\sin)'(t)
			\left(\alpha+\frac{\omega_0}{2}+\beta t\right)\dd t
			=-\beta\int_{t_-}^{t_+}q(t)\sin t\dd t.
		\end{equation}
		Choosing $q(t)=\sin t$ gives
		$\beta\int_{t_-}^{t_+}\sin^2t\dd t=0$, hence $\beta=0$.
		It follows that $\Om=(\alpha,\alpha+\omega_0)\times(t_-,t_+)$.
		The containment $\Om\subset\{x_1>0\}$ gives $\alpha\ge 0$,
		and the point $(0,0)\in\partial\Om$ gives $\alpha=0$.
		The disk $B((d,0),d)$ is contained in the rectangle, so $\omega_0\ge2d$.
		This proves \eqref{eq:rectangle}.
		Finally, set $H=t_+-t_-=(k+l)\pi$. The function
		$\cos(\pi(x_2-t_-)/H)$ has zero mean on the rectangle and a Rayleigh
		quotient $\pi^2/H^2$. Using it in \eqref{eq:rayleigh} gives
		\begin{equation}\label{eq:rectangle-spectral-bound}
			\mu_2(\Om)\le\frac{\pi^2}{H^2}=\frac1{(k+l)^2}\le\frac14.
		\end{equation}
	\end{proof}
	
	\subsection{Proofs of the main results}
	
	\begin{proof}[Proof of Theorem {\ref{thm:rigidity}}]
		Let $\sigma\in\{-1,1\}$ be chosen so that $\sigma u(p)=M$.
		On the dilated domain $\widehat\Om=\sqrt\mu\,\Om$, put
		\begin{equation}\label{eq:rigidity-scaling}
			\widehat u(y)=\frac{\sigma}{M}u(y/\sqrt\mu),
			\qquad\widehat p=\sqrt\mu\,p.
		\end{equation}
		A change of variables in \eqref{eq:weak} shows that $\widehat u$
		is a weak Neumann eigenfunction on $\widehat\Om$ with eigenvalue $1$.
		Moreover,
		\begin{equation}\label{eq:scaled-distance}
			\widehat u(\widehat p)=\max_{\cl{\widehat\Om}}|\widehat u|=1,
			\qquad\dist(\widehat p,\partial\widehat\Om)
			=\sqrt\mu\,\dist(p,\partial\Om).
		\end{equation}
		A rigid motion now gives \eqref{eq:configuration} with
		$d\le\pi/\sqrt3$. Proposition \ref{prop:comparison} provides the maximum
		segment, and Lemma \ref{lem:rectangle} provides the rectangle and
		$\widehat u=\cos y_2$ in these coordinates.
		Undoing the rigid motion makes the $y_2$-direction a unit vector $\tau$;
		undoing the dilation and sign normalization gives
		{\eqref{eq:rigid-function}}. The corresponding side length is
		$(k+l)\pi/\sqrt\mu$, with $m=k+l\ge2$.
		Finally, eigenvalue scaling gives
		\begin{equation}\label{eq:scaled-second-eigenvalue}
			\mu_2(\widehat\Om)=\frac{\mu_2(\Om)}\mu\le\frac14,
		\end{equation}
		which is {\eqref{eq:spectral-factor}}.
	\end{proof}
	\vspace{0.2em}
    
	\begin{proof}[Proof of Theorem {\ref{thm:main}}]
		We first combine the diameter bound of Kr\"oger \cite{Kroger} with the
		localization theorem of Rohleder \cite[Theorem 1.1]{Rohleder}.
		Write $D_\Om=\diam\Om$. With $j_{m,1}$ denoting the first positive
		zero of $J_m$, the diameter bound reads
		\begin{equation}\label{eq:diameter}
			\sqrt{\mu_2(\Om)}D_\Om\le2j_{0,1}.
		\end{equation}
		For any interior critical point $p$ of a second Neumann eigenfunction,
		Rohleder's theorem gives
		\begin{equation}\label{eq:rohleder}
			R_p:=\max_{y\in\partial\Om}|y-p|
			>\frac{j_{1,1}}{2j_{0,1}}D_\Om.
		\end{equation}
		Here, Neumann regularity on bounded convex domains \cite{Grisvard}
		implies \(u\in H^2(\Omega)\), so Rohleder's theorem applies. Neither boundary smoothness
		nor simplicity of $\mu_2(\Om)$ is needed.

		Put $d_p=\dist(p,\partial\Om)$ and choose $y\in\partial\Om$ with
		$|y-p|=R_p$. The ray from $y$ through $p$ meets the boundary again at
		a point $y'$ beyond $p$. Thus
		\begin{equation}\label{eq:diameter-chord}
			D_\Om\ge|y-y'|=R_p+|p-y'|\ge R_p+d_p.
		\end{equation}
		Combining \eqref{eq:diameter}, \eqref{eq:rohleder} and \eqref{eq:diameter-chord} yields
		\begin{equation}\label{eq:localization}
			\sqrt{\mu_2(\Om)}d_p
			<2j_{0,1}-j_{1,1}<\frac\pi{\sqrt3}.
		\end{equation}

		By Proposition \ref{prop:gradient}, $u$ is continuous on $\cl\Om$, so
		$M=\max_{\cl\Om}|u|>0$ is attained. Suppose it is attained at
		$p\in\Om$. Then $u(p)=M$ or $u(p)=-M$, and $\nabla u(p)=0$.
		Equation \eqref{eq:localization} allows us to apply
		Theorem {\ref{thm:rigidity}} with $\mu=\mu_2(\Om)$, giving
		\begin{equation}\label{eq:main-contradiction}
			0<\mu_2(\Om)\le\frac14\mu_2(\Om),
		\end{equation}
		a contradiction. Hence every point where $|u|=M$ lies on the boundary,
		and {\eqref{eq:main}} follows.
	\end{proof}
	\vspace{0.2em}
    
	\iffalse The following consequence of Theorem {\ref{thm:main}} is immediate.	
	\begin{corollary}\label{cor:equal-extrema}
		Let $u$ be a second Neumann eigenfunction on a bounded convex planar domain $\Om$.
		If $\max_{\cl\Om}u=-\min_{\cl\Om}u$, then its maximum and minimum are both
		attained only on $\partial\Om$. In particular, this holds whenever
		$u\circ T=-u$ for an isometry $T$ with $T(\Om)=\Om$, including
		odd reflection symmetry.
	\end{corollary}
	\fi

	\begin{proof} [Proof of Corollary \ref{cor:central}]
    We here give an alternative proof of the antisymmetry using Rohleder's localization theorem \cite{Rohleder}.
		After translating the domain, we may assume $c=0$. For any
		$x\in\Om$, central symmetry gives $-x\in\Om$, and convexity places
		their midpoint $0$ in $\Om$.
		Reflection through the origin preserves both the domain and the weak
		Neumann equation. Thus
		\begin{equation}\label{eq:even-component}
			g(x)=\frac{u(x)+u(-x)}2
		\end{equation}
		is either zero or a second eigenfunction. If it is nonzero, its evenness implies $\nabla g(0)=0$. Equation \eqref{eq:rohleder} then gives a point
		$y\in\partial\Om$ such that
		\begin{equation}\label{eq:central-radius-contradiction}
			|y|>\frac{j_{1,1}}{2j_{0,1}}\diam\Om.
		\end{equation}
		But $-y\in\partial\Om$, so $2|y|\le\diam\Om$. Since
		$j_{1,1}/(2j_{0,1})>1/2$, the two inequalities contradict each other.
		Therefore $g=0$ and $u$ is odd.
		For a nonzero odd function on this domain, then
		$\max_{\cl\Om}u=-\min_{\cl\Om}u=\max_{\cl\Om}|u|$.
		Theorem \ref{thm:main} shows that maximum and minimum of $u$ are attained only on $\partial\Om$.
	\end{proof}

    \vspace{0.2em}
	{\begin{proof}[Proof of Theorem \ref{thm:higher-maximum-modulus}]
			Let $0\ne u\in E_\mu(\Om)$ and set
			$M=\max_{\cl\Om}|u|$. If $|u(p)|=M$ at some $p\in\Om$, then
			\[
			\sqrt\mu\,\dist(p,\partial\Om)
			\le\sqrt\mu\,r_\Om\le\frac{\pi}{\sqrt3}.
			\]
			Theorem {\ref{thm:rigidity}} would force $\Om$ to be a rectangle,
			contrary to the hypothesis. Hence all points where $|u|=M$ lie on
			$\partial\Om$, proving \eqref{eq:higher-maximum-modulus}.
	\end{proof}}
	\vspace{0.2em}
	
	{\begin{proof}[Proof of Corollary \ref{cor:finite-higher-modes}]
			Kr\"oger's diameter estimate \cite[Theorem 1]{Kroger} gives
			\begin{equation}\label{eq:kroger-all-eigenvalues}
				\mu_k(\Om)D_\Om^2
				\le\bigl(2j_{0,1}+(k-2)\pi\bigr)^2,
				\qquad k\ge2.
			\end{equation}
			Consequently, for every $2\le k\le N+1$, the hypothesis
			\eqref{eq:diameter-inradius-criterion} yields
			\[
			\sqrt{\mu_k(\Om)}\,r_\Om
			\le\sqrt{\mu_{N+1}(\Om)}\,r_\Om
			\le\frac{r_\Om}{D_\Om}
			\bigl(2j_{0,1}+(N-1)\pi\bigr)
			\le\frac{\pi}{\sqrt3}.
			\]
			Since $\Om$ is not a rectangle,
			Theorem \ref{thm:higher-maximum-modulus} applies to every nonzero
			element of each eigenspace $E_{\mu_k(\Om)}(\Om)$, proving the assertion.
	\end{proof}}
	\vspace{0.2em}

	{\begin{proof}[Proof of  Theorem \ref{thm:triangle-third-mode}]
			Let
			$a\ge b\ge c>0$ be the side lengths of $T$, so $D=a$ and
			$P=a+b+c$, and set
			\[
			Q=a^2+b^2+c^2,\qquad r=r_T=\frac{2S}{P}.
			\]
			{The side lengths and inradius satisfy}
			\begin{equation}\label{eq:triangle-side-inradius}
				\frac{Q}{P^2}+2\frac{r^2}{D^2}\le\frac12.
			\end{equation}
			Indeed, put $s=P/2$, $x=s-a$, $y=s-b$, and $z=s-c$.
			{Since $a$ is the longest side, $x>0$, $x\le y$, and $x\le z$. Also,} $s=x+y+z$, $D=y+z$, and Heron's formula gives
			$r^2=xyz/s$. Hence
			\[
			\frac12-\frac{Q}{P^2}-2\frac{r^2}{D^2}
			=\frac{xD(y-z)^2+yz(D^2-4x^2)}{2s^2D^2}\ge0,
			\]
			since $D=y+z\ge2x$.
			
			The estimates of
			Laugesen and Siudeja \cite[Theorem 3.5]{LaugesenSiudejaMax} and
			\cite[Theorem 3.3]{LaugesenSiudejaMin} give
			\begin{equation}\label{eq:triangle-spectral-inputs}
				\mu_2(T)+\mu_3(T)\le\frac{2\pi^2Q}{9S^2},
				\qquad \mu_2(T)D^2>j_{1,1}^2.
			\end{equation}
			Here $j_{m,1}$ is the first positive zero of $J_m$.
			Let $t=r^2/D^2$. Multiplying the first inequality by
			$r^2=4S^2/P^2$ and using \eqref{eq:triangle-side-inradius}, we obtain
			\[
			(\mu_2(T)+\mu_3(T))r^2
			\le\frac{8\pi^2}{9}\frac{Q}{P^2}
			\le\frac{4\pi^2}{9}-\frac{16\pi^2}{9}t.
			\]
			Subtracting the lower bound for $\mu_2(T)r^2$ gives
			\begin{equation}\label{eq:triangle-decreasing-bound}
				\mu_3(T)r^2<\frac{4\pi^2}{9}
				-\left(\frac{16\pi^2}{9}+j_{1,1}^2\right)t.
			\end{equation}
			On the other hand, Kr\"oger's bound
			\eqref{eq:kroger-all-eigenvalues} with $k=3$ gives
			\[
			\mu_3(T)r^2\le(2j_{0,1}+\pi)^2t.
			\]
			Set $B=16\pi^2/9+j_{1,1}^2$ and $C=(2j_{0,1}+\pi)^2$.
			{The smaller of these bounds is largest at their intersection,
				$t=4\pi^2/(9(B+C))$. Hence}
			\begin{equation}\label{eq:triangle-uniform-bound}
				\mu_3(T)r^2
				\le\min\left\{Ct,\frac{4\pi^2}{9}-Bt\right\}
				\le\frac{4\pi^2C}{9(B+C)}
				<\frac{\pi^2}{3}.
			\end{equation}
			Since $T$ is not a rectangle, Theorem \ref{thm:higher-maximum-modulus} gives
			\eqref{eq:triangle-third-maximum-modulus} for every
			$0\ne u\in E_{\mu_3(T)}(T)$.

			To prove the vertex conclusion for acute triangles, set
			$M=\max_{\cl T}|u|$. By \eqref{eq:triangle-third-maximum-modulus},
			no interior point of $T$ attains $M$. Suppose that
			$p$ lies on an open side $e$ and satisfies $|u(p)|=M$.
			Reflect $T$ across $e$, writing $R_e$ for this reflection and
			\[
			K_e=\operatorname{int}\bigl(\cl T\cup R_e(\cl T)\bigr).
			\]
			Then $K_e$ is a convex quadrilateral that is not a rectangle.
			The even extension of $u$, denoted $U$, belongs
			to $H^1(K_e)$ and satisfies the Neumann equation with eigenvalue
			$\mu_3(T)$, though $\mu_3(T)$ need not remain the third eigenvalue of $K_e$.
			Moreover, $U$ attains its maximum modulus $M$ at the interior point $p\in K_e$.
			
			Let $E$ and $F$ be the two vertices of $e$, set $\ell=|EF|$, $\sigma=|E-p|$, and let $\alpha,\beta$ denote the
			angles of $T$ at $E$ and $F$, respectively. Then
			\[
			\dist(p,\partial K_e)
			=\min\{\sigma\sin\alpha,(\ell-\sigma)\sin\beta\}
			\le\frac{\ell\sin\alpha\sin\beta}{\sin\alpha+\sin\beta}
			=\frac{2S}{P-\ell}\le\frac{2S}{P-D}=\rho.
			\]
			Thus \eqref{eq:triangle-vertex-criterion} gives
			$\sqrt{\mu_3(T)}\,\dist(p,\partial K_e)\le\pi/\sqrt3$.
			Theorem {\ref{thm:rigidity}} would force $K_e$ to be a rectangle,
			a contradiction. Hence every maximum-modulus point is
			a vertex.
	\end{proof}}
	\vspace{0.2em}

	{\begin{proof}[Proof of Corollary \ref{thm:maximum-modulus-set}]
			Set $M=\max_{\cl\Om}|u|$ and $\mathcal M=\mathcal M_{\mathrm{int}}(u)$.
			In the rectangular case \eqref{eq:maximum-set-rectangular-mode},
			$\mathcal M$ contains the segments
			$(0,W)\times\{jL/m\}$, $1\le j\le m-1$, and is therefore infinite.
			Suppose henceforth that \eqref{eq:maximum-set-rectangular-mode} does not hold.

			If $p\in\mathcal M$ satisfied
			$\sqrt\mu\,\dist(p,\partial\Om)\le\pi/\sqrt3$,
			Theorem {\ref{thm:rigidity}} would make $\Om$ a rectangle and give
			$u(x_1,x_2)=\pm M\cos(\sqrt\mu(x_2-p_2))$ in coordinates with $x_2\in(0,L)$
			parallel to the oscillating direction. The Neumann condition at $x_2=0$
			gives $\sin(\sqrt\mu\,p_2)=0$. Hence $\sqrt\mu\,p_2=k\pi$ for some
			integer $k$, and
			$u(x_1,x_2)=\pm M\cos(\sqrt\mu\,x_2)$.
			The Neumann condition at $x_2=L$ also gives $\sin(\sqrt\mu L)=0$,
			so $L=m\pi/\sqrt\mu$ for a positive integer $m \ge 2$.
			Thus we obtain precisely
			\eqref{eq:maximum-set-rectangular-mode}, a contradiction.
			Thus \eqref{eq:maximum-set-distance-gap} holds. Let
			\[
			\mathcal M=\{x\in\cl\Om:|u(x)|=M,  \dist(x,\partial\Om)\ge \pi/\sqrt{3\mu} \}.
			\]
			By continuity, $\mathcal M$ is compact and contained in $\Om$.

			Suppose that $\mathcal M$ is infinite, and let $\mathcal A$ be its
			nonempty compact set of accumulation points. Fix $p\in\mathcal A$.
			Replacing $u$ by $-u$ if necessary, we may assume $u(p)=M$ and $u>0$ near $p$.
			In this neighborhood, $\mathcal M$ is the set where $u=M$. At $p$,
			\[
			\nabla u(p)=0,\qquad \operatorname{tr}D^2u(p)=\Delta u(p)=-\mu M<0.
			\]
			After a rotation, $\partial_{22}u(p)<0$. The analytic implicit function
			theorem gives a neighborhood $U$ of $p$ and an analytic function
			$g:I\to\R$, with $p_1\in I$, such that
			\[
			\{x\in U:\partial_2u(x)=0\}=\{(x_1,g(x_1)):x_1\in I\}.
			\]
			Every point of $\mathcal M\cap U$ lies on this graph. Since these points
			accumulate at $p$, the analytic function $x_1\mapsto u(x_1,g(x_1))-M$ has
			zeros accumulating at $p_1$ and hence vanishes identically on $I$.
			Consequently,
			\[
			\mathcal M\cap U=\mathcal A\cap U=\{(x_1,g(x_1)):x_1\in I\}.
			\]
			Here every point of the graph is itself an accumulation point of
			$\mathcal M$. Thus $\mathcal A$ is a compact embedded analytic
			one-dimensional manifold without boundary. Let $\Gamma$ be a connected component of $\mathcal A$. Then $\Gamma$ is a simple closed analytic curve.

			Let $G$ be the bounded domain enclosed by $\Gamma$.
			On $\Gamma$ we have $u^2=M^2$ and $\nabla u=0$.
			Consider the vector field
			\[
			X=(x\cdot\nabla u)\nabla u-\frac12|\nabla u|^2x+\frac\mu2u^2x,
			\]
			which satisfies the planar Rellich--Pohozaev identity
			\begin{equation}\label{eq:maximum-set-rellich-divergence}
				\diver X=(x\cdot\nabla u)(\Delta u+\mu u)+\mu u^2=\mu u^2.
			\end{equation}
			Since $X=(\mu M^2/2)x$ on $\Gamma$, the divergence theorem gives
			\begin{equation}\label{eq:maximum-set-rellich-integral}
				\mu\int_Gu^2\dd x
				=\int_\Gamma X\cdot\nu_G\dd s
				=\frac{\mu M^2}{2}\int_\Gamma x\cdot\nu_G\dd s
				=\mu M^2|G|,
			\end{equation}
			where $\nu_G$ is the outward unit normal to $G$.
			Since $u^2\le M^2$, this identity together with continuity forces $u^2=M^2$ on $G$.
			Continuity then implies $u=M$ on $G$,
			which yields a contradiction.
			Hence $\mathcal M$ is finite.
	\end{proof}}

    \vspace{0.2em}
	\begin{remark}\label{rem:sharp-factor}
		The factor $4$ in {\eqref{eq:spectral-factor}} is optimal. On
		$\Om_L=(0,L)\times(-\pi,\pi)$, with $0<L\le2\pi$, the function
		$u(x)=\cos x_2$ is a Neumann eigenfunction with eigenvalue $1$, whereas
		$\mu_2(\Om_L)=\min\{\pi^2/L^2,1/4\}=1/4$.
		Its maximum is attained at every point of $(0,L)\times\{0\}$, including
		points arbitrarily close to the boundary. This establishes sharpness
		of the spectral factor, but not of the distance threshold $\pi/\sqrt3$.
		
		The unit disk gives an upper bound on any possible universal distance
		threshold for this rectangular rigidity. The radial function
		$u(x)=J_0(j_{1,1}|x|)$ is a Neumann eigenfunction with eigenvalue
		$j_{1,1}^2$.
		The integral representation of $J_0$ gives $|J_0(r)|<1$ for $r>0$, while $J_0(0)=1$ \cite{DLMF}.
		Thus its maximum modulus is attained at the center, where
		$\sqrt\mu\,\dist(0,\partial\Om)=j_{1,1}$, although the disk is not
		a rectangle. Consequently, a universal constant replacing
		$\pi/\sqrt3$ in the hypothesis \eqref{eq:threshold} must be strictly less than $j_{1,1}$. These examples do not determine the optimal threshold.
	\end{remark}
	
	\iffalse	
	\begin{remark}\label{rem:one-sided-extrema}{
			Strict boundary maximum modulus does not exclude interior attainment
			of the extremum with smaller absolute value, even arbitrarily close to
			the boundary of a nonrectangular convex domain. Consider the half-disk
			\[
			\Om=\{(x,y)\in\R^2:x^2+y^2<1,\ y>0\},
			\qquad u(x,y)=J_0\bigl(j_{1,2}\sqrt{x^2+y^2}\bigr),
			\]
			where $j_{1,2}$ is the second positive zero of $J_1$.
			The function $u$ is a Neumann eigenfunction corresponding to the tenth eigenvalue
			$\mu=\mu_{10}(\Om)=j_{1,2}^2$. Then
			\[
			|u(x,y)|<1=u(0,0)=\max_{\partial\Om}|u|
			\qquad((x,y)\in\Om).
			\]
			On the other hand, $J_0$ decreases on $(0,j_{1,1})$ and increases
			on $(j_{1,1},j_{1,2})$ \cite{DLMF}. Thus
			\[
			\min_{\cl\Om}u=J_0(j_{1,1})\in(-1,0),
			\]
			and attained at every point of the interior semicircle
			$\{(x,y):x^2+y^2=j_{1,1}^2/j_{1,2}^2,\ y>0\}$.
			Hence no positive universal distance threshold can yield the rigidity
			conclusion of Theorem {\ref{thm:rigidity}} from smaller absolute value side.
		}
	\end{remark}
	\fi

	\vspace{0.1em}
	\begin{example} \label{ex:sector-smaller-extremum}
		The inradius condition \eqref{eq:inradius-criterion} does not exclude
		interior attainment of the extremum with smaller absolute value.
		To see this, consider the convex unit sector
		of opening angle $\pi/6$:
		\[
		\Om=\{(r\cos\theta,r\sin\theta):0<r<1,\ 0<\theta<\pi/6\},
		\qquad u(r,\theta)=J_0(j_{1,2}r).
		\]
		The Bessel equation gives $-\Delta u=j_{1,2}^2u$.
		Since $u$ is radial, its normal derivative vanishes on both straight
		sides; on the circular arc $r=1$, the identity $J_0'=-J_1$ gives
		$\partial_r u=-j_{1,2}J_1(j_{1,2})=0$. Thus $u$ is a Neumann
		eigenfunction with eigenvalue $\mu=\mu_3(\Om)=j_{1,2}^2$.
		In fact, the first three Neumann eigenvalues are
		\[
		\begin{aligned}
			\mu_1(\Om)=0,\quad
			\mu_2(\Om)=j_{1,1}^2,\quad
			\mu_3(\Om)=j_{1,2}^2.
		\end{aligned}
		\]

		The incenter lies on the angle bisector. If its distance from the
		vertex is $a$, tangency to the straight sides and the unit circle gives
		$r_\Om=a\sin(\pi/12)=1-a$. Consequently,
		\[
		r_\Om=\frac{\sin(\pi/12)}{1+\sin(\pi/12)},\qquad
		\mu r_\Om^2
		=j_{1,2}^2\left(\frac{\sin(\pi/12)}{1+\sin(\pi/12)}\right)^2
		<\frac{\pi^2}{3}.
		\]
		The integral representation of $J_0$ gives $|J_0(t)|<1$ for $t>0$.
		Hence
		\[
		|u(x)|<1=u(0)=\max_{\partial\Om}|u|
		\qquad(x\in\Om),
		\]
		and the maximum modulus on $\cl\Om$ is attained only at the boundary
		vertex $0$. On the other hand, $J_0$ decreases on $(0,j_{1,1})$
		and increases on $(j_{1,1},j_{1,2})$. Thus
		\[
		\min_{\cl\Om}u=J_0(j_{1,1})\approx-0.4,
		\qquad
		\left|\min_{\cl\Om}u\right|<1=\max_{\cl\Om}u.
		\]
		This extremum with smaller absolute value is attained at every interior point
		of the circular arc
		\[
		\left\{
		\frac{j_{1,1}}{j_{1,2}}(\cos\theta,\sin\theta):
		0<\theta<\pi/6
		\right\},
		\qquad \frac{j_{1,1}}{j_{1,2}}\approx0.55.
		\]
		It is also attained at the two boundary endpoints of this arc.
		Therefore strict boundary
		maximum modulus does not exclude interior points attaining the
		extremum with smaller absolute value.
	\end{example}

\end{document}